\documentclass[12pt]{article}
\usepackage{amssymb, amscd, amsthm, amsmath,latexsym, amstext,mathrsfs,verbatim}
\usepackage[all]{xy}
\usepackage{graphicx,enumerate, url}
\def\lijntje{\vrule height2.4pt depth-2pt width0.5in}

\def\vlijntje{\vrule height0.45in depth0.4pt width0.4pt}
\def\vlijn{\buildrel {\hbox to 0pt{\hss$\textstyle\circ$\hss}}\over\vlijntje}

\def\dlijntje{{\vrule height2pt depth-1.6pt
width0.5in}\llap{\vrule height4pt depth-3.6pt width0.5in}}
\def\tlijntje{{\vrule height1.7pt depth-1.3pt
width0.5in}\llap{\vrule height3.0pt depth-2.6pt width0.5in}\llap{\vrule height4.3pt depth-3.9pt width0.5in}
}
\def\vtriple#1\over#2\over#3{\mathrel{\mathop{\kern0pt #2}\limits_{\hbox
to 0pt{\hss$#1$\hss}}^{\hbox to 0pt{\hss$#3$\hss}}}}
\def\rvtriple#1\over#2\over#3{\mathrel{\mathop{\kern0pt #2}\limits_{\hbox
to 0pt{\hss$#3$\hss}}^{\hbox to 0pt{\hss$#1$\hss}}}}
\def\Ct{\vtriple{\scriptstyle 2}\over\circ\over{}
\kern-1pt\lijntje\kern-1pt\vtriple{\scriptstyle 1}\over\circ\over{}
\kern-4pt{\dlijntje \kern -25pt<}\kern8pt
\vtriple{\scriptstyle 0}\over\circ\over{}\kern-1pt
}
\def\Bt{\vtriple{\scriptstyle 2}\over\circ\over{}
\kern-1pt\lijntje\kern-1pt\vtriple{\scriptstyle 1}\over\circ\over{}
\kern-4pt{\dlijntje \kern -25pt>}\kern8pt
\vtriple{\scriptstyle 0}\over\circ\over{}\kern-1pt}

\def\ddA{{\rm A}}

\def\det{{\rm det}}

\newcommand{\C}{\mathbb C}

\newcommand{\N}{\mathbb N}

\newcommand{\R}{\mathbb R}
\newcommand{\Z}{\mathbb Z}

\def\Dm{\vtriple{\scriptstyle n+1}\over\circ\over{}\kern-1pt\lijntje\kern-1pt
\vtriple{\scriptstyle{n}}\over\circ\over{}
\cdots\cdots\vtriple{\scriptstyle 4}\over\circ\over{}\kern-1pt\lijntje\kern-1pt
\vtriple{\scriptstyle 3}\over\circ\over{\buildrel
{\scriptstyle 2}\over\vlijn}\kern-1pt\lijntje\kern-1pt
\vtriple{1}\over\circ\over{}\kern-1pt}

\def\Dn{\vtriple{\scriptstyle n}\over\circ\over{}\kern-1pt\lijntje\kern-1pt
\vtriple{\scriptstyle{n-1}}\over\circ\over{}
\cdots\cdots\vtriple{\scriptstyle 4}\over\circ\over{}\kern-1pt\lijntje\kern-1pt
\vtriple{\scriptstyle 3}\over\circ\over{\buildrel
{\scriptstyle 2}\over\vlijn}\kern-1pt\lijntje\kern-1pt
\vtriple{1}\over\circ\over{}\kern-1pt}
\def\En{\vtriple{\scriptstyle n}\over\circ\over{}\kern-1pt\lijntje\kern-1pt
\vtriple{\scriptstyle{n-1}}\over\circ\over{}
\cdots\cdots\vtriple{\scriptstyle 5}\over\circ\over{}\kern-1pt\lijntje\kern-1pt
\vtriple{\scriptstyle 4}\over\circ\over{\buildrel
{\scriptstyle 2}\over\vlijn}\kern-1pt\lijntje\kern-1pt
\vtriple{\scriptstyle 3}\over\circ\over{}\kern-1pt\lijntje\kern-1pt
\vtriple{\scriptstyle 1}\over\circ\over{}\kern-1pt}
\def\An{\vtriple{\scriptstyle n}\over\circ\over{}\kern-1pt\lijntje\kern-1pt
\vtriple{\scriptstyle{n-1}}\over\circ\over{}\kern-1pt\lijntje\kern-1pt
\vtriple{\scriptstyle n-2}\over\circ\over{}
\cdots\cdots
\vtriple{\scriptstyle 2}\over\circ\over{}\kern-1pt\lijntje\kern-1pt
\vtriple{\scriptstyle 1}\over\circ\over{}\kern-1pt}
\def\Cn{\vtriple{\scriptstyle n-1}\over\circ\over{}
\kern-1pt\lijntje\kern-1pt\vtriple{\scriptstyle{n-2}}\over\circ\over{}
\cdots\cdots
\vtriple{\scriptstyle 2}\over\circ\over{}
\kern-1pt\lijntje\kern-1pt\vtriple{\scriptstyle 1}\over\circ\over{}
\kern-4pt{\dlijntje \kern -25pt<}\kern10pt
\vtriple{\scriptstyle 0}\over\circ\over{}\kern-1pt}
\def\Ct{\vtriple{\scriptstyle 2}\over\circ\over{}
\kern-1pt\lijntje\kern-1pt\vtriple{\scriptstyle 1}\over\circ\over{}
\kern-4pt{\dlijntje \kern -25pt<}\kern12pt
\vtriple{\scriptstyle 0}\over\circ\over{}\kern-1pt
}
\def\Bn{\vtriple{\scriptstyle n-1}\over\circ\over{}
\kern-1pt\lijntje\kern-1pt\vtriple{\scriptstyle{n-2}}\over\circ\over{}
\cdots\cdots
\vtriple{\scriptstyle 2}\over\circ\over{}
\kern-1pt\lijntje\kern-1pt\vtriple{\scriptstyle 1}\over\circ\over{}
\kern-4pt{\dlijntje \kern -25pt>}\kern10pt
\vtriple{\scriptstyle 0}\over\circ\over{}\kern-1pt}
\def\Bt{\vtriple{\scriptstyle 2}\over\circ\over{}
\kern-1pt\lijntje\kern-1pt\vtriple{\scriptstyle 1}\over\circ\over{}
\kern-4pt{\dlijntje \kern -25pt>}\kern12pt
\vtriple{\scriptstyle 0}\over\circ\over{}\kern-1pt}
\def\Es{\vtriple{\scriptstyle 6}\over\circ\over{}\kern-1pt\lijntje\kern-1pt
\vtriple{\scriptstyle 5}\over\circ\over{}\kern-1pt\lijntje\kern-1pt
\vtriple{\scriptstyle 4}\over\circ\over{\buildrel
{\scriptstyle 2}\over\vlijn}\kern-1pt\lijntje\kern-1pt
\vtriple{3}\over\circ\over{}\kern-1pt\lijntje\kern-1pt
\vtriple{\scriptstyle 1}\over\circ\over{}\kern-1pt}
\def\Ff{
\vtriple{\scriptstyle 1}\over\circ\over{}
\kern-1pt\lijntje\kern-1pt\vtriple{\scriptstyle 2}\over\circ\over{}
\kern-4pt{\dlijntje \kern -25pt<}\kern10pt
\vtriple{\scriptstyle 3}\over\circ\over{}\kern-1pt\lijntje\kern-1pt
\vtriple{\scriptstyle 4}\over\circ\over{}
\kern-1pt}
\def\Ht{
\vtriple{\scriptstyle 1}\over\circ\over{}
\kern-1pt\overset{5}{\lijntje}\kern-1pt\vtriple{\scriptstyle 2}\over\circ\over{}
\kern-1pt\lijntje\kern-1pt
\vtriple{\scriptstyle 3}\over\circ\over{}\kern-1pt}
\def\Hf{
\vtriple{\scriptstyle 1}\over\circ\over{}
\kern-1pt\overset{5}{\lijntje}\kern-1pt\vtriple{\scriptstyle 2}\over\circ\over{}
\kern-1pt\lijntje\kern-1pt
\vtriple{\scriptstyle 3}\over\circ\over{}\kern-1pt\lijntje\kern-1pt
\vtriple{\scriptstyle 4}\over\circ\over{}
\kern-1pt}
\def\In{
\vtriple{\scriptstyle 0}\over\circ\over{}
\kern-1pt\overset{n}{\lijntje}\kern-1pt\vtriple{\scriptstyle 1}\over\circ\over{}
\kern-1pt}
\def\Gt{
\vtriple{\scriptstyle 0}\over\circ\over{}
\kern-4pt{\tlijntje\kern -25pt<}\kern 10pt\vtriple{\scriptstyle 1}\over\circ\over{}
\kern-1pt}
\def\EBn{\vtriple{\scriptstyle n-1}\over\circ\over{}
\kern-1pt\lijntje\kern-1pt\vtriple{\scriptstyle{n-2}}\over\circ\over{\buildrel
{\scriptstyle -1}\over\vlijn}\cdots\cdots
\vtriple{\scriptstyle 2}\over\circ\over{}
\kern-1pt\lijntje\kern-1pt\vtriple{\scriptstyle 1}\over\circ\over{}
\kern-4pt{\dlijntje \kern -25pt<}\kern8pt
\vtriple{\scriptstyle 0}\over\circ\over{}\kern-1pt}
\def\Cn{\vtriple{\scriptstyle n-1}\over\circ\over{}
\kern-1pt\lijntje\kern-1pt\vtriple{\scriptstyle{n-2}}\over\circ\over{}
\cdots\cdots
\vtriple{\scriptstyle 2}\over\circ\over{}
\kern-1pt\lijntje\kern-1pt\vtriple{\scriptstyle 1}\over\circ\over{}
\kern-4pt{\dlijntje \kern -25pt<}\kern10pt
\vtriple{\scriptstyle 0}\over\circ\over{}\kern-1pt}
\def\ECn{\vtriple{\scriptstyle -2}\over\circ\over{}
\kern-4pt{\dlijntje \kern -25pt>}\kern8pt\vtriple{\scriptstyle n-1}\over\circ\over{}
\kern-1pt\lijntje\kern-1pt\vtriple{\scriptstyle{n-2}}\over\circ\over{}
\cdots\cdots
\vtriple{\scriptstyle 2}\over\circ\over{}
\kern-1pt\lijntje\kern-1pt\vtriple{\scriptstyle 1}\over\circ\over{}
\kern-4pt{\dlijntje \kern -25pt<}\kern12pt
\vtriple{\scriptstyle 0}\over\circ\over{}\kern-1pt}
\def\Fo{\vtriple{\scriptstyle -1}\over\circ\over{}
\kern-1pt\lijntje\kern-1pt
\vtriple{\scriptstyle 1}\over\circ\over{}
\kern-1pt\lijntje\kern-1pt\vtriple{\scriptstyle 2}\over\circ\over{}
\kern-4pt{\dlijntje \kern -25pt<}\kern8pt
\vtriple{\scriptstyle 3}\over\circ\over{}\kern-1pt\lijntje\kern-1pt
\vtriple{\scriptstyle 4}\over\circ\over{}
\kern-1pt}
\def\Ft{
\vtriple{\scriptstyle 1}\over\circ\over{}
\kern-1pt\lijntje\kern-1pt\vtriple{\scriptstyle 2}\over\circ\over{}
\kern-4pt{\dlijntje \kern -25pt<}\kern8pt
\vtriple{\scriptstyle 3}\over\circ\over{}\kern-1pt\lijntje\kern-1pt
\vtriple{\scriptstyle 4}\over\circ\over{}
\kern-1pt\lijntje\kern-1pt
\vtriple{\scriptstyle -2}\over\circ\over{}
\kern-1pt}
\def\Go{\vtriple{\scriptstyle -1}\over\circ\over{}
\kern-1pt\lijntje\kern-1pt
\vtriple{\scriptstyle 0}\over\circ\over{}
\kern-4pt{\tlijntje\kern -25pt<}\kern 12pt\vtriple{\scriptstyle 1}\over\circ\over{}
\kern-1pt}
\def\Gf{
\vtriple{\scriptstyle 0}\over\circ\over{}
\kern-4pt{\tlijntje\kern -25pt<}\kern 12pt\vtriple{\scriptstyle 1}\over\circ\over{}
\kern-1pt\lijntje\kern-1pt
\vtriple{\scriptstyle -2}\over\circ\over{}
\kern-1pt}

\numberwithin{equation}{section}

\newtheorem{lemma}{Lemma}[section]
\newtheorem{cor}[lemma]{Corollary}
\newtheorem{prop}[lemma]{Proposition}
\newtheorem{thm}[lemma]{Theorem}

\theoremstyle{definition}

\newtheorem{defn}[lemma]{Definition}

\theoremstyle{remark}
\newtheorem{rem}[lemma]{Remark}

\def\ad{\mathrm{ad}}

\begin{document}
\title{Characteristic polynomials and finitely dimensional representations of $\mathfrak{sl}(2, \C )$}
\author{Tianyi Jiang, Shoumin Liu\footnote{The author is funded by the NSFC (Grant No. 11971181).}}
\date{}
\maketitle

\begin{abstract}
In this paper, we obtain a  general formula for the characteristic polynomial of  a finitely dimensional  representation of Lie algebra   $\mathfrak{sl}(2, \C )$ and the form for these characteristic polynomials, and
prove there is one to one correspondence  between  representations and their characteristic polynomials. We define a product on these  characteristic polynomials,  endowing them with a monoid structure.

\end{abstract}

\section{Introduction}
The determinant  and eigenvalues of a square matrix are  classical topics in linear algebra.
The interaction of several matrices is an important subject in multilinear algebra. We generalize the characteristic polynomial $\det(ZI-A)$ of  a square matrix $A$ for several matrices, which means that we study
the polynomial $\det(z_0I+z_1A_1+\cdots+z_mA_m)$ for $n\times n$  matrices $A_1$, $A_2$, $\cdots$, $A_m$.  In \cite{Y2009}, Yang defines the projective spectrums through the multiparameters pencil $z_1A_1+\cdots+z_mA_m$,
and there are many fruitful results in \cite{CSZ2016}, \cite{GY2017} and \cite{HY2014}.
For a  representation $\pi$ of  a finite group $G=\{g_i\}_{i=1}^{n}$, the characteristic polynomial $$Q_{\pi}(z_0,z_1,\cdots,z_n)=\det(z_0I+\sum_{i=1}^{n}z_i\pi(g_i))$$  is well studied in \cite{D1969} and  \cite{F1896},
where they prove that the characteristic   polynomial of $\pi$ is irreducible when
$\pi$ is  an irreducible representation of $G$.
It is quite natural to consider similar topics for  Lie algebras of finite dimension. In  \cite{CCD2019}, \cite{HZ2019} and \cite{H2021}, the characteristic polynomial $f_\pi(z_0,z_1,z_2,z_3)=\det(z_0 I +z_1\phi(h)+z_2\phi(e_1)+z_3\phi(e_2))$ in   $\mathfrak{sl}(2, \C )$  on
its irreducible representation $\pi$ of dimension $m+1$ is obtained, which is
\begin{eqnarray}\label{fpi}
 f_\pi=\begin{cases}
 z_0\prod_{l=1}^{m/2}\left(z_0^2-4l^2(z_1^2+z_2z_3)\right) \quad 2\mid m.\\
 \\
 \prod_{l=0}^{(m-1)/2}\left(z_0^2-(2l+1)^2(z_1^2+z_2z_3)\right)\quad 2\nmid m.
 \end{cases}
\end{eqnarray}

  The finitely dimensional  representations  of $\mathfrak{sl}(2, \C )$ are well studied, so it is natural to consider their  characteristic polynomials,
  and describe all these polynomials and
  the corresponding relation between  them.
  Since the representations of $\mathfrak{sl}(2, \C )$ can form a monoidal category, there must be some algebraic structure on their  characteristic polynomials, hence it is possible to endow these polynomials with an algebraic structure.

  Our paper is sketched as the following. In Section \ref{CA}, we present a theorem which can allow us to replace the canonical basis of $\mathfrak{sl}(2, \C )$ under conjugation automorphisms.
  Apply the conclusion from  Section \ref{CA}, we present a new approach to prove the Hu and Zhang's conjecture in Section \ref{HZCONJ}. In Section \ref{Moddec}, we show a clear formula of the characteristic polynomial for  any finitely dimensional representation of $\mathfrak{sl}(2, \C )$,
  and prove there is one to one correspondence between finitely dimensional representations and their characteristic polynomials, and we also  present the forms of all characteristic polynomials. By the tensor product of representations, in Section \ref{monoid} we define the resolution product of two characteristic polynomials, and we show there is a monoid structure on  characteristic polynomials under the resolution product. In Section \ref{adjrep}, we apply our results to the adjoint representation of   $\mathfrak{sl}(n, \C )$ restricted to  $\mathfrak{sl}(2, \C )$ to explore some results about   the characteristic polynomial of  $\mathfrak{sl}(n, \C )$.

\section{Conjugation automorphisms on $\mathfrak{sl}(2, \C )$}\label{CA}
The Lie  algebra  $\mathfrak{sl}(2, \C )$ is a basic object in Lie theory, and it is well known that it has a canonical basis
$$h=\left(
      \begin{array}{cc}
        1 & 0 \\
        0 & -1 \\
      \end{array}
    \right),\quad
 e_1=\left(
      \begin{array}{cc}
        0 & 1 \\
        0 & 0 \\
      \end{array}
    \right), \quad
     e_2=\left(
      \begin{array}{cc}
        0 & 0 \\
        1 & 0 \\
      \end{array}
    \right),$$
 with their Lie bracket relations:
 \begin{equation}
 [e_1,e_2]=h,\quad [h,e_1]=2e_1,\quad [h,e_2]=-2e_2.\label{3rels}
 \end{equation}
By the Lie algebra automorphisms of $\mathfrak{sl}(2, \C )$, there are many possible choices for this kind of basis, the following theorem implies more choices for them.
 \begin{thm}\label{hinsl2}
 Suppose that we have a matrix $h^{'}\in \mathfrak{sl}(2, \C)$ with $\det(h)=-1$, then there exists a $2\times 2$ invertible matrix $A$,
such that
$ AhA^{-1}=h^{'}$. Let $Ae_1A^{-1}=e_1^{'}$, $Ae_2A^{-1}=e_2^{'}$. Therefore, $h^{'}$, $e_1^{'}$, $e_2^{'}$ satisfy the relations (\ref{3rels}) by replacing
 $h$, $e_1$, $e_2$, respectively.
 \end{thm}
 \begin{proof} Suppose that the eigenvalues of $h^{'}$ are $\lambda_1$, $\lambda_2$.
 Since  we have $h^{'}\in \mathfrak{sl}(2, \C)$ with $\det(h)=-1$, it follows that
 $$\lambda_1+\lambda_2=0,\quad \lambda_1\lambda_2=-1,$$
 therefore, we have $\{\lambda_1, \lambda_2\}=\{\pm 1\}.$  So there is  a $2\times 2$ invertible matrix $A$,
such that
$ AhA^{-1}=h^{'}$. Easily, we can check that the conjugation map,
\begin{eqnarray*}
\sigma_{A}: \mathfrak{sl}(2, \C)&\rightarrow& \mathfrak{sl}(2, \C),\\
                x&\mapsto& AxA^{-1}
\end{eqnarray*}
is a Lie algebra automorphism on $\mathfrak{sl}(2, \C)$. Hence the  three elements
 $h^{'}$, $e_1^{'}$, $e_2^{'}$ satisfy the relations (\ref{3rels}) by replacing
 $h$, $e_1$, $e_2$, respectively.
 \end{proof}
 \section{Hu-Zhang's conjecture by conjugation automorphisms}\label{HZCONJ}
 \begin{rem}\label{basis}
 It is well known that  $\mathfrak{sl}(2, \C)$ has an irreducible representation $\phi:\mathfrak{sl}(2, \C)\rightarrow \mathfrak{gl}(V) $,
  where $V$ is a complex vector space of dimension $m+1$  with maximal eigenvalue (highest weight) $m$
for $h$ and $m\geq 0$. And the linear space $V$ has a basis $v_0$, $v_1$, $\cdots$, $v_m$, with $\phi(h)(v_i)=(m-2i) v_i$.
When we replace $h$ by $h^{'}$ in Theorem \ref{hinsl2}, the matrix $\phi(h^{'})$ still has eigenvalues $-m$, $-m+2$, $\cdots$, $m-2$, $m$, since the Lie algebra can be defined by generators and relations(\cite[Chapter 18]{H1972}).
\end{rem}
\begin{defn}
The polynomial $$f_{\varphi}(z_0, z_1, z_2, z_3)=\det(z_0 I +z_1\varphi(h)+z_2\varphi(e_1)+z_3\varphi(e_2))$$
 is called the characteristic polynomial of $\mathfrak{sl}(2, \C)$ on the representation $\varphi$, where $\varphi$ is a finite representation of $\mathfrak{sl}(2, \C)$.
\end{defn}
In \cite{HZ2019}, Hu and Zhang present a conjecture that
$$f_m(z_0, z_1)=f_{\phi}(z_0, z_1, 1, 1)=\prod_{i=1}^{m}\left(z_0-(m-2i)\sqrt{z_1^2+1}\right).$$
The formula plays an important role in showing (\ref{fpi}) in \cite{CCD2019}.
It is proved  by showing the eigenvalues and eigenvectors directly in \cite{CCD2019},  or by computing the eigenvalues of tridiagonal matrices in  \cite{H2021}.
Here we present a proof by avoiding these computations.
\begin{proof} Let
$$h^{'}=\left(
         \begin{array}{cc}
           0 & 1 \\
           1 & 0 \\
         \end{array}
       \right)=e_1+e_2.
$$ Then we can see that $h^{'}$ satisfies the condition in Theorem \ref{hinsl2}, then we can find the $e_1^{'}$ and $e_2^{'}$ for  Theorem \ref{hinsl2}, where

$$e_1^{'}=\left(
         \begin{array}{cc}
           1/2 & -1/2 \\
           1/2& -1/2 \\
         \end{array}
       \right),\quad
       e_2^{'}=\left(
         \begin{array}{cc}
           1/2 & 1/2 \\
           -1/2& -1/2 \\
         \end{array}
       \right).
$$
So $h=e_1^{'}+ e_2^{'}$,  and
$z_1h+e_1+e_2=z_1(e_1^{'}+e_2^{'})+h^{'}$.
Up to some base change on $V$, $h^{'}$, $e_1^{'}$, and $e_2^{'}$ have the same matrices
as $h$, $e_1$,  and $e_2$, respectively. Therefore, it follows that
\begin{equation}\label{fbsc}
f_{\phi}(z_0,z_1,z_2,z_3)=\det(z_0 I +z_1\phi(h^{'})+z_2\phi(e_1^{'})+z_3\phi(e_2^{'}))
\end{equation}
Hence we have  \begin{eqnarray*}
f_m(z_0,z_1)&=&f_{\phi}(z_0,z_1,1,1)\\
&=&\det(z_0 I +z_1\phi(h)+\phi(e_1)+\phi(e_2))\\
&=&\det(z_0 I +z_1\phi(h^{'})+\phi(e_1^{'})+\phi(e_2^{'}))\\
&=&\det(z_0 I +z_1\phi(e_1+e_2)+\phi(h))\\
&=&\det(z_0 I +z_1\phi(e_1)+z_1\phi(e_2)+\phi(h))\\
&=&f_{\phi}(z_0, 1,z_1,z_1)\\
&=& \det\left(z_0 I +\phi(h+z_1(e_1+e_2))\right).
 \end{eqnarray*}
 We have
 $$h+z_1(e_1+e_2)=\left(
                    \begin{array}{cc}
                      1 & z_1 \\
                      z_1 & -1 \\
                    \end{array}
                  \right).
 $$
 It follows that $\det(h+z_1(e_1+e_2))=-1-z_1^2$. When $z_1^2\neq -1$,  we can check that
 $$h^{''}=\frac{h+z_1(e_1+e_2)}{\sqrt{1+z_1^2}}$$
 satisfies the Theorem \ref{hinsl2}. Therefore, by Theorem \ref{hinsl2} and Remark \ref{basis}, we can have that $\phi(h^{''})$ has eigenvalues
 $-m$, $-m+2$, $\cdots$, $m$, which implies that $h+z_1(e_1+e_2)$ has eigenvalues $-\sqrt{1+z_1^2}m$, $-\sqrt{1+z_1^2}(m-2)$, $\cdots$, $\sqrt{1+z_1^2}m$.
 Hence when we consider the $z_0$ and $z_1$ are complex numbers, and $z_1^2\neq -1$, it follows that
$$f_m(z_0,z_1)=\prod_{i=0}^{m}\left(z_0+(m-2i)\sqrt{1+z_1^2}\right),$$
 or \begin{equation}\label{fm}
 f_m(z_0,z_1)=\begin{cases}
 z_0\prod_{l=1}^{m/2}\left(z_0^2-4l^2(1+z_1^2)\right) \quad 2\mid m.\\
 \\
 \prod_{l=0}^{(m-1)/2}\left(z_0^2-(2l+1)^2(1+z_1^2)\right)\quad 2\nmid m.
 \end{cases}
\end{equation}
It is known that  $\C^{2}\setminus \{z_1^2=-1\}$ is dense in $\C^{2}$,  so the conjecture holds by the continuity of polynomial functions.
\end{proof}
\begin{rem} In the proof, we can obtain the conclusion without so many deformations, by considering the eigenvalues of
$$z_1h+(e_1+e_2)=\left(
                  \begin{array}{cc}
                    z_1 & 1 \\
                    1 & -z_1 \\
                  \end{array}
                \right).
$$ When we apply  the formula (\ref{fbsc}), we can get one formula
$$f_{\phi}(z_0,z_1,1,1)=f_{\phi}(z_0,1,z_1,z_1),$$ which can show more symmetric relations among the variables $z_1$ , $z_2$ and $z_3$.
\end{rem}

 \section{Characteristic polynomials of $\mathfrak{sl}(2, \C)$}\label{Moddec}
In \cite{CCD2019}, the authors obtain the  $f_{\phi}(z_0,z_1,z_2,z_3)$ by proving the Hu's conjecture.
Now we consider  a $\mathfrak{sl}(2, \C)$ representation $\phi:\mathfrak{sl}(2, \C)\rightarrow \mathfrak{gl}(V)$, where $V$ is a finite dimension complex vector space.
It is known that  $\phi(h)$ is semisimple (diagonalizable) with integer eigenvalues. For each $n\in \Z$, we use $d_{n,\phi}$ denote the dimension of  eigenvector space $\phi(h)$ of the eigenvalue $n$, namely
$$d_{n,\phi}=\dim\{v\in V\mid \phi(h)v=nv\}.$$
For the representation $\phi$, the following theorem holds.
\begin{thm}\label{reptopol}
Let $\phi:\mathfrak{sl}(2, \C)\rightarrow \mathfrak{gl}(V)$ be a representation of $\mathfrak{sl}(2, \C)$, then
the characteristic polynomial
\begin{equation}\label{fv}
f_{\phi}(z_0,z_1,z_2,z_3)=z_0^{d_{0,\phi}}\prod_{n\geq 1}(z_0^2-n^2(z_1^2+z_2z_3))^{d_{n,\phi}}.
\end{equation}

\end{thm}
\begin{proof}
By (\ref{fpi}), it is easy to verify that the formula (\ref{fv}) holds for the finite dimensional  irreducible representation of $\mathfrak{sl}(2, \C)$.
By the semisimplity of $\mathfrak{sl}(2, \C)$, then we have the decomposition $$\phi\simeq \oplus_{t=1}^{s}\phi_t ,$$ Where each
$\phi_t$ is an irreducible representation of $\mathfrak{sl}(2, \C)$.
By the definition of characteristic polynomial,  we have $$f_{\phi}(z_0,z_1,z_2,z_3)=\prod_{t=1}^{s}f_{\phi_t}(z_0,z_1,z_2,z_3).$$
Furthermore, we have $d_{n,\phi}=\sum_{t=1}^{s}d_{n,\phi_t}$. It follows that
\begin{eqnarray*}
f_{\phi}(z_0,z_1,z_2,z_3)&=&\prod_{t=1}^{s}f_{\phi_t}(z_0,z_1,z_2,z_3)\\
&=&\prod_{t=1}^{s}\left(z_0^{d_{0,\phi_t}}\prod_{n\geq 1}(z_0^2-n^2(z_1^2+z_2z_3))^{d_{n,\phi_t}}\right)\\
&=& z_0^{\sum_{t=1}^{s}d_{0,\phi_t}}\prod_{n\geq 1}(z_0^2-n^2(z_1^2+z_2z_3))^{\sum_{t=1}^{s}d_{n,\phi_t}}\\
&=&z_0^{d_{0,\phi}}\prod_{n\geq 1}(z_0^2-n^2(z_1^2+z_2z_3))^{d_{n,\phi}}.
\end{eqnarray*}
\end{proof}
On the other direction, the structure of  a finite dimensional $\mathfrak{sl}(2, \C)$-module is totally decided by its characteristic polynomial.
\begin{thm}\label{poltorep}
Two finite dimensional  representations $\phi$ and $\psi$ of $\mathfrak{sl}(2, \C)$ are  isomorphic if and only if they have the same characteristic polynomial.
\end{thm}
\begin{proof} The necessarity is obvious. Suppose that  the characteristic polynomial $f_{\phi}(z_0,z_1,z_2,z_3)$ is fixed,
we will show that the representation $\phi$ can be reconstructed in a unique way.
Because the polynomial ring $\C[z_0,z_1,z_2,z_3]$ is a unique factorization ring, by Theorem \ref{reptopol}, we can suppose that
\begin{equation}\label{fN}
f_{\phi}(z_0,z_1,z_2,z_3)=z_0^{d_0}\prod_{n= 1}^{N}(z_0^2-n^2(z_1+z_2z_3))^{d_n},
\end{equation}
with $d_N\geq 1$.
Let $\varphi_m$ be the irreducible representation of $\mathfrak{sl}(2, \C)$ of dimension $m+1$ with highest weight $m$, for $m\geq 0$.
Suppose that $$\phi\simeq \oplus_{m\geq 0} l_m \varphi_m,$$
 where $l_m$ is the multiplicity of $\varphi_m$ in $\phi$. By the  equation  (\ref{fN}), we see that
 $$\phi\simeq \oplus_{m=0}^{N} l_m \varphi_m.$$
For Characteristic polynomials, it is known that $$f_{\phi}=\prod_{m=0}^{N}f_{\varphi_m}^{l_m}.$$  Therefore by the  equation  (\ref{fN}),
we see that $l_N=d_N$.
Therefore $$f_{\oplus_{m=0}^{N-1} l_m \varphi_m}=\frac{f_{\phi}}{f_{\varphi_N}^{d_N}},$$ thus the $l_m$ for $m\leq N-1$ can be decided by induction.
\end{proof}
From the algorithm in the Proof of Theorem \ref{poltorep}, it is easy to obtain the following description for the  characteristic polynomials of all finite dimensional representations of $\mathfrak{sl}(2, \C)$.
\begin{cor} A polynomial $f_(z_0,z_1,z_2,z_3)\in \C[z_0,z_1,z_2,z_3]$ is a characteristic polynomial of a finite dimensional representation of $\mathfrak{sl}(2, \C)$ if and only if
$$f_{\phi}(z_0,z_1,z_2,z_3)=z_0^{d_0}\prod_{n= 1}^{N}(z_0^2-n^2(z_1^2+z_2z_3))^{d_n}$$ with $d_n\geq d_{n+2}$, for any $n\in \N.$
\end{cor}

\section{ A monoid structure on $\mathbf{CP}_{\mathfrak{sl}(2, \C)}$}\label{monoid}
 We denote  characteristic polynomials of all finite dimensional representations of $\mathfrak{sl}(2, \C)$ by $\mathbf{CP}_{\mathfrak{sl}(2, \C)}$ in the Section \ref{Moddec}, which is a subset of $\C[z_0,z_1,z_2,z_3]$.
 In this section, a monoid structure  on $\mathbf{CP}_{\mathfrak{sl}(2, \C)}$ will be defined.
 \begin{defn}\label{respro}
 Let $f_{\phi}$, $f_{\psi}$ be two polynomials in $\mathbf{CP}_{\mathfrak{sl}(2, \C)}$. Considering $z_0=z_0-0\sqrt{z_1^{2}+z_2z_3}$ in Theorem \ref{reptopol},  for two representations $\phi$ and $\psi$,
  we can write
 \begin{eqnarray*}
 f_{\phi}=\prod_{i=1}^{n}\left(z_0+\alpha_i\sqrt{z_1^{2}+z_2z_3}\right),\\
 f_{\psi}=\prod_{j=1}^{m}\left(z_0+\beta_j\sqrt{z_1^{2}+z_2z_3}\right),
 \end{eqnarray*}
 with $\alpha_i$, $i=1,\cdots, n$ being all the eigenvalues of $\phi(h)$ and
  $\beta_j$, $j=1,\cdots, m$ being all the eigenvalues of $\psi(h)$.
Define $f_{\phi} * f_{\psi}\in \C[\sqrt{z_1^{2}+z_2z_3},z_0,z_1,z_2,z_3]$ by
 \begin{equation}
 f_{\phi} * f_{\psi}=\prod_{i.j=1}^{n,m}\left(z_0+(\alpha_i+\beta_j)\sqrt{z_1^{2}+z_2z_3}\right)
 \end{equation}
 and we call $ f_{\phi} * f_{\psi}$ the \emph{resolution product} of   $f_{\phi}$ and $f_{\psi}$.
  \end{defn}
  \begin{prop} \label{tenprod}
  Let $f_{\phi}$, $f_{\psi}$ be two polynomials in $\mathbf{CP}_{\mathfrak{sl}(2, \C)}$ as in the Definition \ref{respro}. It follows that
   $$f_{\phi} * f_{\psi}=f_{\phi\otimes \psi}.$$
  \end{prop}
  \begin{proof}

  Suppose that $\{v_i\}_{i=1}^{n}$ is a basis of representation $\phi$ such that  $$\phi(h)(v_i)=\alpha_i v_i,$$ and
    $\{w_j\}_{j=1}^{m}$ is a basis of representation $\psi$ such that  $$\psi(h)(w_j)=\beta_j w_j.$$
    Let $\lambda_{v_i}=\alpha_i$ and $\lambda_{w_j}=\beta_j$, for $1\leq i\leq n$ and $1\leq j\leq m$.
    It is known that $\{v_i\otimes w_j\}_{i=1,j=1}^{n,m}$ is a basis of $\phi\otimes \psi$. Furthermore, we have
    \begin{eqnarray*}
    &&\phi\otimes \psi(h)(v_i\otimes w_j)\\
    &=&\phi(h)\otimes I+ I\otimes \psi(h)(v_i\otimes w_j)\\
    &=&\phi(h)(v_i)\otimes w_j+v_i\otimes \psi(h)(w_j)=(\alpha_i+\beta_j)(v_i\otimes w_j).
    \end{eqnarray*}
    Let $\lambda_{v_i\otimes w_j}=\alpha_i+\beta_j$, for  $1\leq i\leq n$ and $1\leq j\leq m$.
    By $\lambda_{v_i}$,  $\lambda_{w_j}$ and $\lambda_{v_i\otimes w_j}$, we can rewrite the $f_{\phi}$, $f_{\psi}$  and $f_{\phi\otimes \psi}$ as the follows,
 \begin{eqnarray*}
 f_{\phi}&=&\prod_{i=1}^{n}\left(z_0+\lambda_{v_i}\sqrt{z_1^{2}+z_2z_3}\right),\\
 f_{\psi}&=&\prod_{j=1}^{m}\left(z_0+\lambda_{w_j}\sqrt{z_1^{2}+z_2z_3}\right),\\
 f_{\phi\otimes \psi}&=&\prod_{i.j=1}^{n,m}\left(z_0+\lambda_{v_i\otimes w_j}\sqrt{z_1^{2}+z_2z_3}\right)\\
 &=&\prod_{i.j=1}^{n,m}\left(z_0+(\alpha_i+\beta_j)\sqrt{z_1^{2}+z_2z_3}\right)= f_{\phi} * f_{\psi}.
 \end{eqnarray*}
 \end{proof}
 \begin{thm}
 The set $\mathbf{CP}_{\mathfrak{sl}(2, \C)}$ is a  commutative monoid under the resolution product with the unit element $z_0$.
 \end{thm}
 \begin{proof}
  By Proposition \ref{tenprod},  the set $\mathbf{CP}_{\mathfrak{sl}(2, \C)}$ is closed under the resolution product.
   For three representations $\phi$, $\psi$, $\varphi$ of $\mathfrak{sl}(2, \C)$, we have
 $\phi\otimes (\psi\otimes \varphi)\simeq (\phi\otimes \psi)\otimes \varphi$. By Proposition \ref{tenprod},
 it follows that
 $$f_{\phi}* (f_\psi*f_\varphi)=f_{\phi\otimes (\psi\otimes \varphi)}=f_{ (\phi\otimes \psi)\otimes \varphi}=(f_{\phi}* f_\psi)*f_\varphi.$$
   Let $\varphi_0$ denote one dimensional trivial representation of  $\mathfrak{sl}(2, \C)$, then $f_{\varphi_0}=z_0$.
   For each representation $\phi$ of  $\mathfrak{sl}(2, \C)$, we have  $\phi\simeq \phi\otimes \varphi_0\simeq \varphi_0\otimes \phi$,  so it follows that
   $f_{\phi}=f_{\phi}*z_0=z_0*f_{\phi}$. Similarly,  the equation $f_{\phi}* f_\psi=f_\psi*f_{\phi}$  holds for $\phi\otimes \psi\simeq \psi\otimes \phi$.
 \end{proof}
 Let $\varphi_m$ be the irreducible representation of $\mathfrak{sl}(2, \C)$ of dimension $m+1$ with highest weight $m$, for $m\geq 0$. For $0\leq n\leq m$,
 it is well known that
 $$\varphi_m\otimes \varphi_n\simeq \oplus_{k=0}^{n}\varphi_{m-n+2k}.$$  By  Proposition \ref{tenprod}, the corollary below holds.
 \begin{cor}\label{tensordecomp}
  For $0\leq n\leq m$, we have
 $$f_{\varphi_m}* f_{\varphi_n}=\prod_{k=0}^{n}f_{\varphi_{m-n+2k}}.$$
 \end{cor}
 \begin{rem} Under the tensor product, the category of  finitely dimensional representations of $\mathfrak{sl}(2, \C)$ is a monoidal category, hence it is natural that the set $\mathbf{CP}_{\mathfrak{sl}(2, \C)}$ should have a monoid structure.
 Considering  that the  $\det(A\otimes I_m-I_n\otimes B)$ is the resolution of characteristic polynomials of $A$ and $B$, where $A$ is a $n\times n$ matrix, and $B$ is a $m\times m$ matrix,  we name the product  resolution product. There are some definitions of
 resolution or determinant of multivariable polynomials in algebraic geometry, so there might be some necessary research for their relations from the view of   algebraic geometry.
 \end{rem}
 \section{The adjoint representation of  $\mathfrak{sl}(2, \C)$ on  $\mathfrak{sl}(n, \C)$}\label{adjrep}
 Let us  recall the canonical basis of $\mathfrak{sl}(n, \C)$ first. Suppose that $e_{ij}$ is  the complex $n\times n$ matrix, with entry $1$ at the $i$th row and $j$th column and  $0$ otherwise.
 The  $\mathfrak{sl}(n, \C)$ is the simple Lie algebra of type $\ddA_{n-1}$ with a basis
 $h_i=e_{ii}-e_{i+1,i+1}$ for $1\leq i\leq n-1$, and $e_{ij}$ for $1\leq i\neq j\leq n$.
 \begin{thm} \label{2ton}
 Let $\phi:\mathfrak{sl}(2, \C)\rightarrow \mathfrak{sl}(n, \C)$ be the Lie algebra homomorphism  defined by $\phi(h)=h_1$, $\phi(e_1)=e_{12}$, $\phi(e_2)=e_{21}$.
 Suppose that $\mathrm{ad}\circ \phi$ be the composition of $\phi$ and the adjoint representation $\mathrm{ad}$ of $\mathfrak{sl}(n, \C)$. Then
\begin{equation}
f_{\mathrm{ad} \circ\phi}(z_0,z_1,z_2,z_3)=z_0^{n^2-5n+6}(z_0^2-(z_1^2+z_2z_3))^{2n-4}(z_0^2-4(z_1^2+z_2z_3))
\end{equation}
 \end{thm}
 \begin{proof} By Theorem \ref{reptopol},  we need to compute the eigenvalues of $\mathrm{ad}h_1$ and their multiplicities. \\
 We recall the root system  $\Phi$ of type $\ddA_{n-1}$. The set $\Phi$ can be realized in $\R^n$, with $\Phi=\{\epsilon_j-\epsilon_i\mid 1\leq i\neq j\leq n\}$ with $\{\epsilon_i\}_{i=1}^{n}$ being the canonical orthonormal basis of
 $\R^n$. For $\mathrm{ad}h_1$, we have
 $$[h_1,h_i]=0, \quad for \quad 1\leq i\leq n-1$$
 $$[h_1, e_{ij}]=
 (\epsilon_j-\epsilon_i,\epsilon_2-\epsilon_1)e_{ij}\quad for \quad 1\leq i\neq j\leq n,$$
 where $(\epsilon_j-\epsilon_i,\epsilon_2-\epsilon_1)$ is the canonical inner product of $\epsilon_j-\epsilon_i$ and $\epsilon_2-\epsilon_1$.
 Hence we get the eigenvalues of $\mathrm{ad}h_1$ are $2$, $-2$, $1$, $-1$ and $0$ with their multiplicities
 $1$, $1$, $2n-4$,  $2n-4$ and $n^2-5n+6$, respectively. Therefore, the theorem holds for  Theorem \ref{reptopol}.
 \end{proof}
 In \cite{KY2021}, the characteristic polynomial of the Lie algebra  $\mathfrak{sl}(n, \C)$ is defined as $Q_{\mathfrak{sl}(n, \C)}(z)$,  the determinant
 of the linear pencil $z_0 I+z_1\ad h_1+\cdots+z_{n-1}\ad h_{n-1} +z_{12} \ad e_{12}+ \cdots+z_{21}\ad e_{21}+\cdots+z_{n-1,n}\ad  e_{n,n-1}$. By the Theorem \ref{2ton},  the following holds.
 \begin{cor} \label{equalsl2}
 For the adjoint representation of  $\mathfrak{sl}(n, \C)$,Let
 $$f_i(z_0,z_i, z_{i,i+1},z_{i+1,i})=\det(z_0 I+z_i \ad h_i+ z_{i,i+1} \ad e_{i,i+1}+ z_{i+1,i}  \ad e_{i+1,i}).$$
 Then it follows that
$$f_i(z_0,z_i, z_{i,i+1},z_{i+1,i})==f_{\mathrm{ad} \circ\phi}(z_0,z_i, z_{i,i+1}, z_{i+1,i}).$$
 \end{cor}
 \begin{proof} We can apply the analogous argument by constructing a Lie algebra isomorphism  $\phi_i$ from $\mathfrak{sl}(2, \C)$  to the
 subalgebra  of $\mathfrak{sl}(n, \C)$ generated by $h_i$, $e_{i,i+1}$ and $e_{i+1,i}$.  By \cite[Theorem 2.3]{KY2021}, there is another way to prove this.  From \cite[Theorem 2.3]{KY2021}, it is known that the characteristic polynomial $Q_{\mathfrak{sl}(n, \C)}(z)$ is invariant under the Lie algebra  automorphism.
 In the Lie algebra, the image of $\phi_i$ is corresponding to the simple root $\alpha_i=\epsilon_{i+1}-\epsilon_{i+1}$ as in the proof of Theorem \ref{2ton}, namely,
 $\phi(\mathfrak{sl}(2, \C))=\C h_{\alpha_i}+\C e_{\alpha_i}+\C e_{-\alpha_i}$. Under the action  Weyl group $W$ of  type $\ddA_{n-1}$ associated to  $\mathfrak{sl}(n, \C)$, all these roots are on the same orbit of $W$. So it implies that there exists a $\sigma\in W$, such that
 $\sigma(\alpha_i)=\alpha_j$ for $1\leq i,j \leq n-1$.   It can be interpreted as a  Lie algebra automorphism of $\mathfrak{sl}(n, \C)$ such that $$\sigma(h_{\alpha})=h_{\sigma(\alpha)},\quad\sigma(e_\alpha)=e_{\sigma(\alpha)}.$$
 Therefore, the conclusion follows just by considering the case in Theorem \ref{2ton}.
\begin{rem} From  \cite[Theorem 2.3]{KY2021} and  the proof of the Corollary \ref{equalsl2}, for a complex simple Lie algebra $L$, its characteristic polynomial is invariant under  its Weyl group, hence the invariant theory of  Weyl groups  might be helpful to compute the characteristic polynomial.
\end{rem}
 \end{proof}

Tianyi jiang\\
Email: jtyoo2021@163.com\\
School of Mathematics, Shandong University\\
Shanda Nanlu 27, Jinan, \\
Shandong Province, China\\
Postcode: 250100\\
Shoumin Liu\\
Email: s.liu@sdu.edu.cn\\
School of Mathematics, Shandong University\\
Shanda Nanlu 27, Jinan, \\
Shandong Province, China\\
Postcode: 250100

\end{document}